\documentclass[a4paper,12pt]{amsart}

\usepackage{amsmath,amsfonts,amssymb,amsthm}
\usepackage{mathtools,mathdots}
\usepackage[hidelinks]{hyperref}
 
\theoremstyle{plain}

\newtheorem{thm}{Theorem}[section]

\newtheorem{lem}[thm]{Lemma}
\newtheorem{prop}[thm]{Proposition}

\newtheorem{thmABC}{Theorem}

\newtheorem{corABC}[thmABC]{Corollary}

\theoremstyle{definition}

\newtheorem{defn}[thm]{Definition}
 \newtheorem{rmk}[thm]{Remark}
\numberwithin{equation}{section}

\newcommand{\norm}[1]{{\lvert #1 \rvert}}
\newcommand{\wt}[1]{\widetilde{#1}}
\newcommand{\wh}[1]{\widehat{#1}}

\newcommand{\cwr}{\mbox{\textnormal{\small\textcircled{$\wr$}}}}

\newcommand{\mS}{\mathcal{S}}

\newcommand{\Sym}{\mathrm{Sym}}
\newcommand{\comm}[1]{}

\title[Subgroup growth of IWPPAs]{On subgroup growth of  iterated wreath \\ products in product action}

\thanks{The author is supported by the Italian program Rita Levi Montalcini for young researchers, Edition 2021.}

\begin{document}


\author[Matteo Vannacci]{Matteo Vannacci}
\address{Dipartimento di Matematica `Ulisse Dini', Universit\`a degli Studi di Firenze, Viale Morgagni, 67/a - 50134 Firenze, Italy}
\email{matteo.vannacci@unifi.it}

\begin{abstract}
We show that there are hereditarily just infinite groups of any subgroup growth type between $n$ and $n^{\log n}$. This is obtained calculating the subgroup growth type of a family of hereditarily just infinite profinite groups obtained via iterated wreath products of finite permutation groups with respect to product actions.
\end{abstract}

\maketitle

\section{Introduction}
\subsection{State of the art}
Let $G$ be a finitely generated profinite group, let $n$ be a positive integer and let $s_n(G)$ denote the number of open subgroups of index at most $n$ in $G$. The sequence $(s_n(G))_{n\in \mathbb{N}}$ has been widely studied; see for instance \cite{segal:finiteimages,ls:subgroupgrowth,barnea_schlagepuchta}. Given a function $f:\mathbb{N}\to \mathbb{R}$, we say that $G$ has \emph{subgroup growth type} $f$, if there are constants $b,c>0$ such that $s_n(G)\le f(n)^b$ for all sufficiently large $n$ and $s_n(G) \ge f(n)^c$ for infinitely many $n$. 

% The main hope in introducing this concept was to obtain a classification of profinite groups ``by growth type''. This led to a beautiful theory and one of its main results is the characterization of the profinite groups with \emph{polynomial subgroup growth}\footnote{i.e.\ finitely generated profinite groups with growth type $n$.}: a finitely generated profinite group has polynomial subgroup growth if and only if it is virtually soluble of finite rank.
  
A natural question that one might ask is whether all positive and increasing functions $f:\mathbb{N}\to \mathbb{R}$ appear as the subgroup growth type of a certain profinite group. This question was referred to as the `Subgroup Growth Gap Problem'. The answer to this question is substantially `yes' and it was settled in two parts, using very different constructions, by Dan Segal in \cite{segal:finiteimages} and L\'aszl\'o Pyber in \cite{pyber}. Segal's answer is based on the construction of suitable \emph{iterated wreath products w.r.t.\ imprimitive actions} with specified subgroup growth type. In this short note we prove some analogous results for certain \emph{hereditarily just infinite profinite groups} that are obtained as inverse limits of iterated wreath products w.r.t.\ product actions. We recall that a profinite group $G$ is just infinite if $G$ is infinite and every non-trivial closed normal subgroup $N$ of $G$ is open.  The group $G$ is \emph{hereditarily just infinite} if every open subgroup $H$ of $G$ is just infinite.

Hereditarily just infinite groups are still a very mysterious family and several natural questions remain unanswered, see for instance \cite[Problems 20.110 and 20.111]{kourovkanotebook}.

One of the most studied classes of hereditarily just infinite groups is that of infinitely iterated wreath products in product action (\cite{matteo,matteoben1,matteoben2}). These groups arise as follows. For $A\le \mathrm{Sym}(\Omega)$ and $B\le \mathrm{Sym}(\Delta)$, denote by $A\cwr B\le \mathrm{Sym}(\Omega^\Delta)$ the wreath product w.r.t.\ product action of $A$ by $B$. Let $\mS = (S_k)_{k\in \mathbb{N}}$, with $S_k \le \Sym(\Omega_k)$, be a sequence of finite transitive permutation groups.  The inverse limit
\[
W^\mathrm{pa}(\mS) = \varprojlim W^\mathrm{pa}_n
\]
of the inverse system $W^\mathrm{pa}_1 \twoheadleftarrow W^\mathrm{pa}_2 \twoheadleftarrow \ldots$ of finite iterated wreath products w.r.t.\ product actions
\[
  W^\mathrm{pa}_n  = S_n \,\cwr\, (S_{n-1} \,\cwr\, ( \cdots \,\cwr\,
  S_1 )) \le
                    \Sym(\wh{\Omega}_n) \quad \text{for $\wh{\Omega}_n =
                      \Omega_n^{\big(\Omega_{n-1}^{\big(\iddots^{\Omega_0}\big)} \big)}$}. 
\]
is called the \emph{infinitely iterated wreath product of type $\mS$ w.r.t.\ product actions}.
By \cite[Theorem~6.2]{reid:characterization} and
\cite{quick:probabilisticgeneration}, every infinitely iterated
wreath product w.r.t.\ product actions $W^\mathrm{pa}(\mS)$, based on a
sequence $\mS$ of finite non-abelian simple permutation groups, is a $2$-generated hereditarily just infinite profinite group that is not virtually pro-$p$ for any prime~$p$. These groups have received some attention in the last few years due to their ``exotic'' properties; see \cite{matteo,matteoben1,matteoben2}.
\subsection{Main result} 
We say that a function $f:\mathbb{N}\rightarrow \mathbb{R}$ is \emph{gently growing} if $f$ is unbounded, non-decreasing, positive and there exist a constants $A,N>0$ such that 
\[
  f(x^{\log x}) \le A f(x) \quad \text{for all $x\ge N$.}
\]
Our main result is the following (cfr.~\cite[Thm.~3]{segal:finiteimages}). 

\begin{thmABC}\label{thm:A}
 Let $f:\mathbb{N}\rightarrow \mathbb{R}$ be a gently growing function. Then there exists a sequence $(p_k)_{k\in \mathbb{N}}$ of primes such that the infinitely iterated wreath product of type $\mathcal{P} = \{\mathrm{PSL}_2(\mathbb{F}_{p_k})\}_{k\in \mathbb{N}}$ w.r.t.\ product actions, with $\mathrm{PSL}_2(\mathbb{F}_{p_k})$ acting on the projective line over $\mathbb{F}_{p_k}$, has subgroup grow type $n^{f(n)}$.
\end{thmABC}

In proving the theorem we also find an explicit expression for the subgroup growth of an infinitely iterated wreath product w.r.t.\ product actions. In particular, Theorem~\ref{thm:A} shows that hereditarily just infinite groups can have a wide spectrum of subgroup growth types.

\begin{corABC}
 Let $f:\mathbb{N}\rightarrow \mathbb{R}$ be a gently growing function. Then there exists a $2$-generated hereditarily just infinite profinite group with subgroup growth type $n^{f(n)}$.
\end{corABC}

To conclude, we present two variations that deal with some non-gently growing functions, see Remark~\ref{rml:variations}. Finally, see Remark~\ref{rmk:exp_growth} for some observations about hereditarily just infinite pro-$p$ groups of exponential growth. 

\section{Nice sequences of groups} \label{sec:prelim}

We are going to prove Theorem~\ref{thm:A} in a slightly more general form. To do so, we will work with `nice sequences of groups'. We denote by $\mathrm{rk}(G)$ %=\max\{\mathrm{d}(H) \vert H\le G\}$
the \emph{rank} of the finite group $G$, i.e.\ the maximum among the minimal number of generators of all subgroups of $G$. Moreover, we write $\mu(G)$ for the minimal index of any proper subgroup in $G$. For instance, for the alternating group $\mathrm{Alt}(5)$, we have $\mu(\mathrm{Alt}(5))=5$.

\begin{defn}
 Let $r$ and $t$ be positive real numbers. A sequence of finite groups $(S_k)_{k\in \mathbb{N}}$ is said \emph{nice with constants $r$ and $t$} if:
\begin{enumerate}  
  \item[N.$1$]  $\norm{S_k} \ge \norm{S_{k-1}}$ for all $k$;
  \item[N.$2$]  $\mathrm{rk}(S_k) \le r$ for all $k$;
  \item[N.$3$]  there is $E_k\le S_k$ elementary abelian such that $\norm{S_k} \le \norm{E_k}^t$ for all $k$; %$S_k$ contains an elementary abelian subgroup $E_k$ such that $\norm{S_k} \le \norm{E_k}^t$ for all $k$; 
  \item[N.$4$]  %if $\mu_k$ is the minimal index of any proper subgroup in $S_k$, then 
  $\mu(S_k) \ge \mu(S_{k-1})$  and $\norm{S_k} \le \mu(S_k)^t$ for all $k$;
  \item[N.$5$]  $\lim_{k\rightarrow \infty} \mu(S_k) = \infty$.
   \end{enumerate}
\end{defn}

\begin{rmk}\label{rmk:psl2}
    Let $(p_k)_{k\in \mathbb{N}}$ be a strictly increasing sequence of primes.  Then the sequence $(\mathrm{PSL}_2(p_k))_{k\in \mathbb{N}}$ is nice with $r=2$ and $t=3$.
\end{rmk}

In fact, $\mathrm{PSL}_2(p_k)$ has rank $2$, $\mu(\mathrm{PSL}_2(p_k))= p_k+1$ and, defining $\tau_k = \norm{\mathrm{PSL}_2(p_k)}$ and $\mu_k= \mu(\mathrm{PSL}_2(p_k))$, we have that $$p_k <\tau_k =\frac{p_k(p_k^2-1)}{2} < \norm{E_k}^3 < \mu_k^3,$$
where $E_k$ is the subgroup of upper unitriangular matrices in~$\mathrm{PSL}_2(p_k)$.

\section{Proof of Theorem~\ref{thm:A}}
We start with a lemma about the extremely fast growth of the ``iterated exponentiation'' of a series of integers.

\begin{lem}\label{lem:ancillary}
Let $(a_k)_{k\in \mathbb{N}}$ be a sequence of positive integers with $a_k\ge 2$ for all $k$. Set $\widehat{a}_0=1$ and $\widehat{a}_n= a_n^{\widehat{a}_{n-1}}$ for $n\ge 1$. Then there exists $N\in \mathbb{N}$ such that $$\sum_{j=0}^n \widehat{a}_j \le 3\, \widehat{a}_n \quad \text{for every $n\ge N$.}$$
\begin{proof}
First, we claim that for every real constant $C>0$ there exists a natural number $M(C)$ such that
 \begin{equation}\label{eq:ak}
 C\, \widehat{a}_{n-1} \le \widehat{a}_n \text{\quad for every $n\ge M(C)$}.
 \end{equation}
It is clear that $(\widehat{a}_k)_{k\in \mathbb{N}}$ is increasing and tends to infinity. In particular, so does $(2^{\widehat{a}_k}/\widehat{a}_k)_{k\in \mathbb{N}}$. This fact, together with the assumption that $a_k\ge 2$ for every $k$, readily implies \eqref{eq:ak}. 

Using \eqref{eq:ak}, we can now show that 
 \begin{equation}\label{eq:sumlessthan2}
  \sum_{j=M(2)}^{n} \widehat{a}_j \le 2\, \widehat{a}_{n} \text{\quad for all $n\ge M(2)$}.
 \end{equation}
 This will be proved by induction on $n$. It is clear that $\widehat{a}_{M(2)}\le 2\, \widehat{a}_{M(2)}$. Suppose by inductive hypothesis that $\sum_{j=M(2)}^{n-1} \widehat{a}_j \le 2\, \widehat{a}_{n-1}$. By \eqref{eq:ak}, $$\sum_{j=M(2)}^n \widehat{a}_j \le 2\, \widehat{a}_{n-1} + \widehat{a}_n\le 2\, \widehat{a}_n,$$ for $n\ge M(2)$. 
 To conclude the proof, it is sufficient to set $N=\max\{M(2)+1,M(M(2))\}$. In fact, in virtue of \eqref{eq:ak}, $M(2)\, \widehat{a}_{M(2)}\le M(2)\, \widehat{a}_{n-1}\le \widehat{a}_n$ for all $n\ge N$. Therefore, by \eqref{eq:sumlessthan2}, $$\sum_{j=0}^{n} \widehat{a}_j \le M(2)\, \widehat{a}_{M(2)} + 2\, \widehat{a}_n \le 3\, \widehat{a}_n \text{\quad for every $n\ge N$}.  \qedhere$$ 
\end{proof}
\end{lem}

Next we bound the subgroup growth of infinitely iterated wreath products in product action from above.

\begin{lem}\label{lem:subgroupgrowth1}
 Let $\mS = (S_k)_{k\in \mathbb{N}}$ be a nice sequence of non-trivial permutation groups $S_k\le \mathrm{Sym}(\Omega_k)$ with constants $r$ and $t$. If $l,n\in \mathbb{N}$ are such that $n<\mu(S_{l+1})$ then $$s_n(W^\mathrm{pa}(\mS))\le n^{3rt\norm{\wh{\Omega}_l}}.$$
\begin{proof}
 % In virtue of Lemma~\ref{lem:ancillary}, there exists $N\in\mathbb{N}$ such that 
%  \[
%    \sum_{i=1}^n \norm{\wh{\Omega}_i} = \sum_{i=1}^{N-1} \norm{\wh{\Omega}_i} + \sum_{i=N}^n \norm{\wh{\Omega}_i}  \le (N-1) \norm{\wh{\Omega}_{N-1}} + 2 3\wt{m}_k
%  \]
% for every $k$ large enough.

Let $l$ and $n$ be as in the hypotheses. We will first show that 
\begin{equation}\label{eq:s_n}
  s_n(W^\mathrm{pa}_{i}) = s_n(W^\mathrm{pa}_{i-1}) = \ldots = s_n(W^\mathrm{pa}_{l+1}) \text{\quad for $i>l$.}
\end{equation}
In fact, by hypothesis, $\mu(S_i) \ge \mu(S_{l+1}) >n$ for $i>l$. Hence, for every $a\in \mathbb{N}$, the direct power $S_i^{a}$ has no proper subgroup of index $n$ or less. Therefore every subgroup of index less to $n$ in $W^\mathrm{pa}_{i}$ contains the base subgroup $S_i^{\norm{\wh{\Omega}_{i-1}}}$ and the claim follows. In particular, \eqref{eq:s_n} yields that $s_n(W^\mathrm{pa}(\mS))= s_n(W^\mathrm{pa}_{l+1})$. To conclude the proof it is then sufficient to estimate the subgroup growth of $s_n(W^\mathrm{pa}_{l+1})$.

Since the permutation groups in the sequence are non-trivial, we have that $\norm{\Omega_k}\ge 2$ for every $k$ and we can apply Lemma~\ref{lem:ancillary} to the sequence $(\norm{\wh{\Omega}_k})_{k\in \mathbb{N}\cup \{0\}}$ with $\norm{\wh{\Omega}_0}=1$. Now, $W^\mathrm{pa}_{l+1}$ has a subnormal series of length $\sum_{j=0}^l \norm{\wh{\Omega}_j} < 3\norm{\wh{\Omega}_l}$. By \cite[Proposition~1.9.1]{ls:subgroupgrowth}, we deduce that 
\[
  s_n(\wt{S}_{l+1}) \le n^{3rt\norm{\wh{\Omega}_l}}
\]
and the lemma is proved.
\end{proof} 
 \end{lem}
 
% Let us denote by $s(G)$ the total number of subgroups of the finite group~$G$.

Finally, we bound the subgroup growth of infinitely iterated wreath products in product action from below.
 
 \begin{lem}\label{lem:subgroupgrowth2}
  Let $\mS = (S_k)_{k\in \mathbb{N}}$ be a nice sequence of permutation groups $S_k\le \mathrm{Sym}(\Omega_k)$ with constants $r$ and $t$. For $n= \norm{S_{l+1}}^{3\norm{\wh{\Omega}_l}}$ we have $s_n(W^\mathrm{pa}(\mS)) \ge n^{\norm{\wh{\Omega}_l}/12t}$.
\begin{proof}
Set $\norm{\wh{\Omega}_0}=1$. By definition of $W^\mathrm{pa}_{l+1}$ and by hypothesis, it is clear that 
 \[
   \norm{W^\mathrm{pa}_{l+1}} = \prod_{j=0}^l \norm{S}_{j+1}^{\norm{\wh{\Omega}_{j}}} \le \norm{S_{l+1}}^{\sum_{j=0}^l \norm{\wh{\Omega}_{j}}} < \norm{S_{k+1}}^{3\norm{\wh{\Omega}_{l}}} = n,
 \]
where in the last inequality we applied Lemma~\ref{lem:ancillary}.
 On the other hand $W^\mathrm{pa}_{l+1}$ contains $S_{l+1}^{\norm{\wh{\Omega}_l}}$. By hypothesis, $S_{l+1}^{\norm{\wh{\Omega}_l}}$ in turn contains 
the elementary abelian subgroup $E_{l+1}^{\norm{\wh{\Omega}_l}}$. Suppose $E_{l+1} = C_p^e$, then, again by hypothesis, $p^{et}\ge \norm{S_{l+1}}$. Moreover, $E_{l+1}^{\norm{\wh{\Omega}_l}}$ has at least $p^{e^2 \norm{\wh{\Omega}_l}^2 /4}$ subgroups. Since $\norm{W^\mathrm{pa}_{l+1}} < n$, we have that $s_n(W^\mathrm{pa}_{l+1})$ %= s(\wt{S}_{l+1})$
is equal to the total number of subgroups of the finite group~$\wt{S}_{l+1}$. Therefore
\begin{multline*}
  s_n(W^\mathrm{pa}(\mS)) \ge s_n(W^\mathrm{pa}_{l+1})  \ge p^{e^2 \norm{\wh{\Omega}_l}^2 /4} \ge \norm{S_{l+1}}^{\norm{\wh{\Omega}_l}^2/4t} = n^{\norm{\wh{\Omega}_l}/12t}.  \qedhere
\end{multline*}
\end{proof}
  \end{lem}

%\section{Proof of Theorem~\ref{thm:A}}
The following immediate consequence of Lemma~\ref{lem:subgroupgrowth1} and Lemma~\ref{lem:subgroupgrowth2} might be of independent interest.
\begin{prop}
Let $\mS = (S_k)_{k\in \mathbb{N}}$ be a nice sequence of finite permutation groups $S_k\le \mathrm{Sym}(\Omega_k)$. Then $W^\mathrm{pa}(\mS)$ has subgroup growth type $n^{\norm{{\wh{\Omega}_{l(n)}}}}$ where $l(n) = \min \{t\in \mathbb{N}\ \vert\ n< \mu(S_{t+1})\}$. 
%\begin{proof}
%Immediate from .
%\end{proof}
\end{prop}

 Note that the function $l(n)$ is defined for every $n$, because of the definition of nice sequence. In particular, we determined the subgroup growth type of every infinitely iterated wreath product w.r.t.\ product actions obtained from a nice sequence of permutation groups. The proof of Theorem~\ref{thm:A} will be completed after we show that we can choose the sequence of simple groups to achieve a specific subgroup growth rate. We first need a technical number-theoretic lemma.

\begin{lem}\textnormal{(\cite[Lemma~13.3.3]{ls:subgroupgrowth})}\label{lem:prime}
Let $f:\mathbb{N}\rightarrow \mathbb{R}$ be a gently growing function. Then there exist constants $B,C\in \mathbb{R}$ with $B,C>0$ such that for every integer $m\ge C$ there exists a prime $p>m$ with 
\[
  f(p) \ge 2m \ge Bf(p^m).
\]
\end{lem}

We are now ready for the proof of Theorem~\ref{thm:A}.

\begin{proof}[Proof of Theorem~\ref{thm:A}]
Let $f$ be a gently growing function. Let $B$ and $C$ be the constants given by Lemma~\ref{lem:prime} for $f$. As in \cite[Section~13.3]{ls:subgroupgrowth} define a sequence of primes as follows. Let $p_0\ge \max\{18,C\}$ be a prime such that $f(p_0)\ge 18$. Suppose that we chose primes $p_0,\ldots,p_{k-1}$. For $i<k$, put $\Omega_i=\mathbb{P}^1(\mathbb{F}_{p_i})$ for the projective line over $\mathbb{F}_{p_i}$, so that $\norm{\Omega_i}=p_i+1$.

As in Remark~\ref{rmk:psl2}, we consider the nice sequence of finite simple groups $(\mathrm{PSL}_2(p_k))_{k\in \mathbb{N}}$ with constants $r=2$ and $t=3$. Moreover, we consider $\mathrm{PSL}_2(p_k)$ as a subgroup of $\mathrm{Sym}(\Omega_k)$ with its natural action. Recall that we set $\tau_k=\norm{\mathrm{PSL}_2(p_k)}$ and $\mu_k=\mu(\mathrm{PSL}_2(p_k))$. Finally, note that $\norm{\widehat{\Omega}_k} = \widehat{m}_k$ with our notation.

By Lemma~\ref{lem:prime}, there exists a prime $p_k$ with $ p_k> 9\wh{m}_{k-1}$ such that
\begin{equation}
 f(p_k) \ge 18 \wh{m}_{k-1} \ge B f(p_k^{9\wh{m}_{k-1}}) \text{\quad for $k\ge 1$}.
\end{equation}

For $n= p_k^{9\wh{m}_{k-1}}$ with $k\ge 1$, using Lemma~\ref{lem:subgroupgrowth2}, Lemma~\ref{lem:prime} and the fact that $p_k^3>\tau_k>p_k$, we have that
\[
  s_n(W^\mathrm{pa}(\mathcal{P})) \ge \left( \tau_k^{3\widehat{m}_{k-1}} \right)^{\widehat{m}_{k-1}/36} \ge 
  %  \tau_k^{\frac{\widehat{m}_{k-1}^2}{36}} = \left( \tau_k^{3\widehat{m}_{k-1}} \right)^{\frac{\widehat{m}_{k-1}}{108}} \ge
  \left( p_k^{9\widehat{m}_{k-1}} \right)^{18\widehat{m}_{k-1}/36\cdot 3\cdot 18} \ge n^{cf(n)}
\]
with $c= B/1944$.

On the other hand, for every $n > p_0$ there is $k\ge 1$ such that $\mu_{k-1} \le n < \mu_k$. Note that $p_k < \mu_k \le n$ and $f$ is non-decreasing. Using Lemma~\ref{lem:subgroupgrowth1} and again Lemma~\ref{lem:prime}, we see that
\[
  s_n(W^\mathrm{pa}(\mathcal{P})) \le n^{18\widehat{m}_{k-1}} \le  n^{f(p_k)} \le n^{f(n)}.
\]
 Therefore $W^\mathrm{pa}(\mathcal{P})$ has the required subgroup growth type. 
\end{proof}

The following remark follows the lines of a similar argument after the proof of \cite[Thm.~13.3.4]{ls:subgroupgrowth}.

\begin{rmk}\label{rml:variations}
 Note that the function $f(n)=(\log n)^\varepsilon$, with $0 < \varepsilon < 1$, is not gently growing. We now mention two variations of the above construction, that allow for a wider range of growth types but for which we have less precise information.

\smallskip

\noindent \emph{Variation 1:} Let $f$ be any unbounded function. Then the sequence of primes $(p_k)_{k\in \mathbb{N}}$ can be chosen so that $s_n(W^\mathrm{pa}(\mathcal{P})) \le n^{f(n)}$ for all large $n$, while $W^\mathrm{pa}(\mathcal{P})$ does not have polynomial subgroup growth. It suffices to ensure that $p_k \ge p_{k-1}^{9\wh{m}_{ k-1}}$ and $f(p_k) \ge 18\wh{m}_{k-1}$ for each $k \ge 1$.

\smallskip

\noindent \emph{Variation 2:} Let $h$ be a positive integer, put $f(n) = (\log n)^{1/h}$ and $f_\ast(n) =(\log n)^{1/(h+1)}$. Then the sequence of primes $(p_k)_{k\in \mathbb{N}}$ can be chosen so that
$s_n(W^\mathrm{pa}(\mathcal{P})) \le n^{f(n)}$ for all large $n$, $s_n(W^\mathrm{pa}(\mathcal{P})) \ge n^{c f_\ast(n)}$
for infinitely many $n$, where $c$ is a positive constant. Again, it suffices to ensure
that $(18\wh{m}_{k-1})^h \le \log p_k \le 2 \cdot (18\wh{m}_{k-1})^h$ for each $k \ge 1$. 

Let us work through the details of Variation 2: if $(18\wh{m}_{k-1})^h \le \log p_k$, then $18\wh{m}_{k-1} \le f(p_k)$, which is what we used in the upper bound. On the other hand, 
\begin{multline*}
  \log p_k \le 2 \cdot (18\wh{m}_{k-1})^h  \Longleftrightarrow (9\wh{m}_{k-1})^h \log p_k \le (18 \wh{m}_{k-1})^{h+1} \Longleftrightarrow \\ \log (p_k^{9\wh{m}_{k-1}}) \le (18 \wh{m}_{k-1})^{h+1} \Longleftrightarrow f_\ast(p_k^{9\wh{m}_{k-1}}) \le 18 \wh{m}_{k-1}
\end{multline*}
which is the inequality used in the lower bound in the proof of Theorem~\ref{thm:A}.
\end{rmk}

\begin{rmk}\label{rmk:exp_growth}
  We point out that there are hereditarily just infinite pro-$p$ groups of a very different kind (\cite[Thm.~1.7]{ershov_jaikin:pwd}), these are quotients of pro-$p$ groups of \emph{positive weighted deficiency} (PWD pro-$p$ groups for short). By \cite[Thm.~B.1]{ershov_jaikin:kazhdan}, PWD pro-$p$ groups have exponential subgroup growth type and we suspect that it is possible to construct hereditarily just infinite pro-$p$ groups with exponential subgroup growth in this way. 
\end{rmk}

\def\cprime{$'$}

\end{document}